\documentclass[12pt]{article}
\usepackage{amsmath, amssymb, amsthm, enumerate, booktabs, accents, framed, graphicx, array, color, multirow, mathrsfs, amsrefs, mathrsfs, cases, euscript, bm, multicol, caption, subfigure}

\usepackage[utf8]{inputenc}
\usepackage[colorlinks,citecolor=blue,urlcolor=blue]{hyperref}
\usepackage[usenames,dvipsnames,svgnames,table]{xcolor}
\usepackage{a4wide}

\allowdisplaybreaks[4]

\numberwithin{equation}{section} \theoremstyle{plain}
\newtheorem{theorem}{Theorem}[section]
\newtheorem{lemma}{Lemma}[section]
\newtheorem{corollary}{Corollary}[section]
\newtheorem{proposition}{Proposition}[section]

\def\bC{\mathbb C}
\def\bE{\mathbb E}
\def\bN{\mathbb N}
\def\bR{\mathbb R}
\def\bP{\mathbb P}

\def\y{\mathbf y}

\def\Tr{\mathrm {Tr}}
\def\diag{\mathrm {diag}}

\begin{document}

\title{On spectrum of sample correlation matrices from large fold tensor vectors}

\date{\today}
\author{Wangjun Yuan\thanks{Department of Mathematics, Southern University of Science and Technology. E-mail: ywangjun@connect.hku.hk}}
\maketitle

\begin{abstract}
In this paper, we investigate the limiting spectral distribution of the sample correlation matrix, whose sample vectors are $k$-fold tensor products of $n$-dimensional vectors with i.i.d. entries. We focus on the limiting regime $n,k \to \infty$ with $k = o(n)$, and we show that the limiting spectral distribution is the Mar\v{c}enko-Pastur law.
As a consequence, we show that the limiting spectral distribution of the Whishart matrix from the $k$-fold tensor product of independent uniformly distributed unit vectors in $\mathbb C^n$ is the Mar\v{c}enko-Pastur law.
\end{abstract}


\medskip 


\section{Introduction}
\label{sec:intro}

For $n \in \bN$, let $\{\xi_1, \ldots, \xi_n\}$ be a family of i.i.d. centered random variables with finite variance. Let $\y = (\xi_1, \ldots, \xi_n) \in \bC^n$ be a random vector and $\{\y_{\alpha}^{(l)}: 1 \le \alpha \le m, 1 \le l \le k\}$ be a family of i.i.d. copies of $\y$.
Here, the notation $\otimes$ is the tensor outer product.
For $1 \le \alpha \le m$, we define the $k$-fold tensor product by $Y_{\alpha} = \y_{\alpha}^{(1)} \otimes \cdots \otimes \y_{\alpha}^{(k)}$, which is a multi-dimensional array with entry
\begin{align*}
    \left( Y_\alpha \right)_{(j_1,\ldots,j_k)} = \prod_{l=1}^k \left( \y_\alpha^{(l)} \right)_{j_l}, \quad
    \forall 1 \le j_1, \ldots, j_k \le n.
\end{align*}
We refer readers to \cite{Tensorbook2018} for more details of tensor outer product.
We identify each $Y_{\alpha}$ as an $n^k$-dimensional vector, and we consider the normalized vector $Y_\alpha/\|Y_\alpha\|$, where $\|Y_\alpha\|$ is the $L^2$-norm of the vector $Y_\alpha$. We set $Y = (Y_1/\|Y_1\|, \ldots, Y_m/\|Y_m\|)$, which is a $n^k \times m$ matrix. Let $\{\tau_1, \tau_2, \ldots\}$ be a sequence of positive real numbers, and $\Lambda_m = \diag(\tau_1,\ldots,\tau_m)$ be the diagonal matrix. We consider the following $n^k \times n^k$ Hermitian matrix
\begin{align}\label{eq-correlated matrix}
	M_{n,k,m} = Y \Lambda_m Y^*
    = \sum_{\alpha=1}^m \tau_{\alpha} \dfrac{Y_{\alpha} Y_{\alpha}^*}{\|Y_\alpha\|^2},
\end{align}
which is the sum of $m$ independent rank-$1$ Hermitian matrices.

For an $N \times N$ matrix $A$ with eigenvalues $\lambda_1,\ldots,\lambda_N$, we define the \emph{empirical spectral distribution} (ESD) function of $A$ by
\begin{align*}
    F^A(x) = \dfrac{1}{N} \sum_{i=1}^N {\bf 1}_{\lambda \le x}, \quad x \in \bR.
\end{align*}
We are interested in the large dimensional limit of $F^{M_{n,k,m}}$.
Throughout the paper, we always assume that the dimensional parameters $m,k$ are functions of $n$, and grow towards infinity following the proportion
\begin{align} \label{eq-def-ratio}
	\lim_{n \to \infty} \dfrac{m}{n^k} \to c
\end{align}
for some constant $c \in (0,\infty)$.
We also assume that the sequence $\{\tau_1,\tau_2,\ldots\}$ satisfies the following limiting moment condition
\begin{align} \label{eq-def-tau}
    \lim_{m \to \infty} \dfrac{1}{m} \sum_{j=1}^m \tau_j^q = m_q^{(\tau)} < \infty
\end{align}
with $|m_q^{{(\tau)}}| \le A^qq^q$ for some positive constant $A$, for all $q \in \bN$.

We are interested in the limiting ESD function $F^{M_{n,k,m}}$ of \eqref{eq-correlated matrix} when the dimension $n \to \infty$.
Our study of model \eqref{eq-correlated matrix} is motivated by problems in both statistics and quantum information theory.

In statistics, it is interesting to understand the unknown population distribution from large samples. The theory to address this problem is hypothesis testing. 
The asymptotic behavior of large samples is derived under certain hypotheses on the population.
Random matrix theory is a very important tool in multivariate statistics, and the spectral properties in high-dimensional cases are of particular interest.

The application of random matrix theory to hypothesis testing dates back to at least \cite{MP67}, where the population covariance matrix is investigated via the sample covariance matrix.
Let's introduce the \emph{sample covariance matrix} as follows:
\begin{align}\label{eq-covariance matrix}
	\widetilde M_{n,k,m} = \widetilde Y \Lambda_m \widetilde Y^*
    = \dfrac{1}{n^k} \sum_{\alpha=1}^m \tau_{\alpha} Y_{\alpha} Y_{\alpha}^*,
\end{align}
where $\widetilde Y = \frac{1}{\sqrt{n}^k}  (Y_1,\ldots,Y_m)$.
It is an $n^k \times n^k$ Hermitian matrix.
We would like to remark that the model \eqref{eq-covariance matrix} is also known as Wishart matrix in some literature.
The paper \cite{MP67} considers the simplest case $k=1$, and derives the limiting empirical spectral distribution (LSD) under the setting that $\xi_1$ has unit variance. The LSD is known as the Marc\v{e}nko-Pastur law. In the case of $\tau_\alpha \equiv 1$, the Marc\v{e}nko-Pastur law has the density function
\begin{align} \label{eq-MP law}
    p_{MP}(x) = \dfrac{\sqrt{((1+\sqrt{c})^2-x)(x-(1-\sqrt{c})^2)}}{2\pi x} {\bf 1}_{[(1-\sqrt{c})^2,(1+\sqrt{c})^2]}(x) + (1-c) \delta_0(dx) {\bf 1}_{0<c<1}.
\end{align}
Further details can be found in \cites{Anderson2010, Bai2010}.
Subsequently, numerous works on sample covariance matrices have emerged in the literature, including \cites{BaiZhou08,Pajor09,Silv95}, and the references therein.
Eventually, for the case $k=1, \tau_\alpha \equiv 1$, the necessary and sufficient condition for the Mar\v{c}enko-Pastur law to serve as the limiting distribution of $F^{\widetilde M_{n,k,m}}$ is established in \cite{Yaskov2016}.

In addition to the sample covariance matrix, the sample correlation matrix, which reveals the correlation on the population, is another fundamental tool in multivariate statistics.
The model \eqref{eq-correlated matrix} we are interested in is known as the \emph{sample correlation matrix}, and it provides information of Pearson
correlation for multivariate data.
The study of the sample correlation matrix \eqref{eq-correlated matrix} can be dated back to \cite{Jiang2004}, where the case $k=1, \tau_\alpha \equiv 1$ is considered. The LSD of \eqref{eq-correlated matrix} is still the Marc\v{e}nko-Pastur law \eqref{eq-MP law}.
Since then, a rich literature on the sample correlation matrices has developed, including \cites{GHPY2017,Ding2020,YLTZ2022,YZZ2023,HY2022,BHXZ2024} and the references therein.
Under the same setting of $k=1, \tau_\alpha \equiv 1$, the necessary and sufficient condition for the convergence of $F^{M_{n,k,m}}$ to the Mar\v{c}enko-Pastur law in the isotropic model was established in \cite{DH2025} if the underlying vector $Y_\alpha$ is a linear combination of a random vector with i.i.d. entries, and in \cite{DongYao2025} for the general isotropic model.

The cases $k \ge 2$ are significantly different from the case $k=1$. The tensor structure appears when $k \ge 2$. Thus, in statistics, the theory with $k \ge 2$ may be applied to the samples from a population with tensor structure.

The second motivation to study the models \eqref{eq-correlated matrix} and \eqref{eq-covariance matrix} comes from quantum information theory.
We start with a classical ball-to-bin probability problem.
We distribute $m$ balls randomly into $n$ bins and see how many balls are in each bin. The problem can be rephrased as the spectrum of $\widetilde M_{n,1,m}$ when $Y_\alpha$ is chosen randomly from the canonical basis $\{e_1, \ldots, e_n\}$ of $\bC^n$ and $\tau_\alpha \equiv 1$.
The canonical basis corresponds to the bins, while $Y_\alpha$ is understood as the $\alpha$-th ball.
We refer interested readers to \cite{Yuan2024} for more details.
The quantum analog of this classical problem is the spectrum of $\widetilde M_{n,1,m}$ if the vector $Y_\alpha$ is picked randomly from the unit sphere $S^{2n-1}$ on $\bC^n$.
A modified version of the quantum problem, which was introduced in \cite{Hastings2012}, is to consider the matrix $M_{n,k,m}$ with the vector $Y_1$ from the random product states in $(\bC^n)^{\otimes k}$. When $k$ and $m/n^k$ are fixed, and the base vector $\y$ is Gaussian, \cite{Hastings2012} established the convergence in expectation of the empirical spectral distribution (ESD) towards the Mar\v{c}enko-Pastur law \eqref{eq-MP law} as $n \to +\infty$ by computing the moments.

In the realm of quantum physics and quantum information theory, it is natural to explore systems with a multitude of quantum states, i.e. consider the case that $k$ is large. See, for example, \cites{Collins2011,Collins2013}.
Moreover, in mathematics, the tensor structure of $Y_\alpha$ leads to a dependence on the entries of $Y_\alpha$ and non-unitary invariance. This phenomenon becomes strong if $k$ is large.

Recently, the LSD for the $k$-fold tensor model \eqref{eq-covariance matrix} was explored in \cite{Lytova2018} for real random vector $\y$ when $k$ is large with the constraint $k = o(n)$. In the special case where $\tau_{\alpha} \equiv 1$, the LSD is exactly the Mar\v{c}enko-Pastur law \eqref{eq-MP law}. The fluctuation of the ESD around the LSD is also established for a class of linear spectral statistics following the approach of \cite{LytovaPastur09} when $k=2$.
In a more recent study, \cite{BYY2022} considered the complex case of $\y$ with $k = O(n)$ for the model \eqref{eq-covariance matrix}. Under the limiting condition $k/n \to d$, the limiting moment sequence of the ESD of $M_{n,k,m}$ is carried out by assuming the finiteness of all moments of $\xi_1$.
It is interesting that the limiting moment sequence obtained in \cite{BYY2022} involves the fourth moment of $\xi_1$ if $d>0$ and $|\xi_1|>1$ with positive probability. If $\tau_{\alpha} \equiv 1$, the resulting limiting moment sequence corresponds to the Mar\v{c}enko-Pastur law \eqref{eq-MP law} when either $d=0$ or $d>0$ and $|\xi_1|=1$ a.s..
In the setting that $\xi_1$ is complex and $|\xi_1|=1$ a.s., the restriction of the growth of $k$ is removed in \cite{Yuan2024}. The LSD of the matrix \eqref{eq-covariance matrix} is established, and it is exactly the Mar\v{c}enko-Pastur law \eqref{eq-MP law} whenever $\tau_{\alpha} \equiv 1$.

The paper focuses on the tensor sample correlation matrix. The number $k$ of the tensor fold is large, with the following limiting constraint.
\begin{align} \label{eq-lim-k}
    \lim_{n \to \infty} \dfrac{k}{n} = 0.
\end{align}
To the best of our knowledge, the results in the present paper are the first on the limiting spectrum of the sample correlation matrix from large fold tensor vectors ($k \gg 1$).

The following theorem is the main result of the paper.

\begin{theorem} \label{Thm}
Let the limiting conditions \eqref{eq-def-ratio}, \eqref{eq-def-tau} and \eqref{eq-lim-k} hold. Assume that for all $p \in \bN$, the $p$-th moment $m_p = \bE [|\xi_1|^p] < \infty$. Then there exists a probability distribution function $F_c$, such that
\begin{align*}
    \lim_{n \to \infty} F^{M_{n,k,m}} = F_c
\end{align*}
in probability.
Moreover, in the case that $\tau_\alpha \equiv 1$, the limiting distribution function $F_c$ is exactly the Mar\v{c}enko-Pastur law with density function \eqref{eq-MP law}.
\end{theorem}

We would like to mention that the two models \eqref{eq-correlated matrix} and \eqref{eq-covariance matrix} are normalized in different ways, and \eqref{eq-correlated matrix} is a case of the quantum analogous problem, noting that the columns of $Y$ are unit vectors, while the columns of $\widetilde Y$ are not.
Thus, the results from the model $\widetilde M_{n,k,m}$ provide information of the quantum problem.
By choosing $\xi_1$ to be a standard Gaussian random variable, we obtain the following corollary.

\begin{corollary} \label{Coro}
Let the limiting conditions \eqref{eq-def-ratio}, \eqref{eq-def-tau} and \eqref{eq-lim-k} hold. Assume that the base random vector $\y$ is uniformly distributed on the unit sphere $S^{2n-1}$ of $\bC^n$. We set $Y' = (Y_1,\ldots,Y_m)$ and define
\begin{align*}
    M_{n,k,m}' = Y' \Lambda_m Y'^* = \sum_{\alpha=1}^m \tau_\alpha Y_\alpha Y_\alpha^*.
\end{align*}
Then there exists a probability distribution function $F_c$, such that
\begin{align*}
    \lim_{n \to \infty} F^{M_{n,k,m}'} = F_c
\end{align*}
in probability.
Moreover, in the case that $\tau_\alpha \equiv 1$, the limiting distribution function $F_c$ is exactly the Mar\v{c}enko-Pastur law with density function \eqref{eq-MP law}.
\end{corollary}

The rest of the paper is organized as follows. We introduce some results on limiting spectrum of the model \eqref{eq-covariance matrix} and some preliminaries on matrices in Section \ref{sec-preliminary}. Then we develop the proofs of Theorem \ref{Thm} and Corollary \ref{Coro} in Section \ref{sec-proofs}.

\section{Preliminaries} \label{sec-preliminary}

The following result is a combination of \cite{BYY2022}*{Proposition 2.4} and the Carleman’s condition.

\begin{lemma} (\cite{BYY2022}*{Proposition 2.4}) \label{Lem-LSD-M}
Let the assumptions in Theorem \ref{Thm} hold. Then there exists a probability distribution function $F_c$, such that
\begin{align*}
    \lim_{n \to \infty} F^{\widetilde M_{n,k,m}} = F_c
\end{align*}
almost surely.
Moreover, in the case that $\tau_\alpha \equiv 1$, the limiting distribution function $F_c$ is exactly the Mar\v{c}enko-Pastur law with density function \eqref{eq-MP law}.
\end{lemma}

Next, we turn to the matrix theory. The following lemma is a simple extension of \cite{Jiang2004}*{Lemma 2.1}.
\begin{lemma} (\cite{Jiang2004}*{Lemma 2.1}) \label{Lem-matrix}
Let $A$ be an $n \times p$ complex random matrix. Denote by $A_j$ the $j$-th column of $A$, and $A_{ij}$ the $(i,j)$-th entry of $A$. Let $B$ be the matrix whose $j$-th column is $A_j/\|A_j\|$. Then for any $n \times n$ diagonal matrix $\Lambda = \diag (\Lambda_{11},\ldots,\Lambda_{nn})$
\begin{align*}
    \Tr \left( \left( \dfrac{1}{\sqrt{n}} A - B \right) \Lambda_n \left( \dfrac{1}{\sqrt{n}} A - B \right)^* \right)
    = \sum_{j=1}^p \Lambda_{jj} \left( \dfrac{1}{n} \|a_j\|^2 - 1 \right) - 2 \sum_{j=1}^p \Lambda_{jj} \left( \dfrac{1}{\sqrt{n}} \|a_j\| - 1 \right).
\end{align*}
\end{lemma}

We end this section by the following difference inequality, which could be found in \cite{Bai1999}.
\begin{lemma} (\cite{Bai1999}*{Lemma 2.7}) \label{Lem-preliminary-compare}
Let $A,B$ be two $p \times n$ complex matrices. Then
\begin{align*}
    L^4 \left( F^{AA^*}, F^{BB^*} \right) \le \dfrac{2}{p^2} \Tr \left( (A-B)(A-B)^* \right)  \Tr \left( AA^*+BB^* \right).
\end{align*}
Here, $L(\cdot,\cdot)$ denotes the L\'evy distance of two distribution functions.
\end{lemma}

\section{Proofs} \label{sec-proofs}

In this section, we develop the proof of Theorem \ref{Thm} using the strategy of comparing $M_{n,k,m}$ with $\widetilde M_{n,k,m}$. Then we explain how to deduce Corollary \ref{Coro} from Theorem \ref{Thm}.

Without loss of generality, we assume that $m_2 = \bE [|\xi_1|^2] = 1$.
We start with the following moment estimates on the norm of the random vector $Y_\alpha$.
\begin{lemma} \label{Lem-moment-Y}
Assume that $m_4 = \bE[|\xi_1|^4] < \infty$, then we have
\begin{align*}
    \bE \left[ \|Y_\alpha\|^2 \right]
    = n^k,
    \quad \quad
    \bE \left[ \|Y_\alpha\|^4 \right]
    = n^{2k} \left( 1 + \dfrac{m_4-1}{n} \right)^k.
\end{align*}
\end{lemma}

\begin{proof}
As $\{\xi_1,\ldots,\xi_n\}$ are i.i.d. centered random variables with unit variance, we have
\begin{align} \label{eq-y-2 moment}
    \bE \left[ \|\y\|^2 \right]
    = \sum_{i=1}^n \bE \left[ |\xi_i|^2 \right]
    = n,
\end{align}
and
\begin{align} \label{eq-y-4 moment}
    \bE \left[ \|\y\|^4 \right]
    = \sum_{i=1}^n \bE \left[ |\xi_i|^4 \right] + \sum_{i\not=j,i,j=1}^n \bE \left[ |\xi_i|^2 |\xi_j|^2 \right]
    = n m_4 + n(n-1)
    = n^2 \left( 1 + \dfrac{m_4-1}{n} \right).
\end{align}
Note that for $1 \le \alpha \le m$, we have
\begin{align} \label{eq-||Y||2}
    \|Y_\alpha\|^2
    = \sum_{j_1,\ldots,j_k=1}^n \prod_{l=1}^k \left| \left( \y_\alpha^{(l)} \right)_{j_l} \right|^2
    = \prod_{l=1}^k \sum_{j_l=1}^n \left| \left( \y_\alpha^{(l)} \right)_{j_l} \right|^2
    = \prod_{l=1}^k \| \y_\alpha^{(l)} \|^2.
\end{align}
By independence, \eqref{eq-y-2 moment}, \eqref{eq-y-4 moment} and \eqref{eq-||Y||2}, we obtain
\begin{align*}
    \bE \left[ \|Y_\alpha\|^2 \right]
    = \prod_{l=1}^k \bE \left[ \| \y_\alpha^{(l)} \|^2 \right]
    = n^k,
\end{align*}
and
\begin{align*}
    \bE \left[ \|Y_\alpha\|^4 \right]
    = \prod_{l=1}^k \bE \left[ \| \y_\alpha^{(l)} \|^4 \right]
    = n^{2k} \left( 1 + \dfrac{m_4-1}{n} \right)^k.
\end{align*}
\end{proof}

The following result is the comparison of the distribution functions of $M_{n,k,m}$ and $\widetilde M_{n,k,m}$, which is the key part of the proof to Theorem \ref{Thm}.

\begin{proposition} \label{Prop-compare}
Let the limiting conditions \eqref{eq-def-ratio}, \eqref{eq-def-tau} and \eqref{eq-lim-k} hold.
Assume that $m_4 = \bE[|\xi_1|^4] < \infty$, then
\begin{align*}
    \lim_{n \to \infty} L \left( F^{M_{n,k,m}}, F^{\widetilde M_{n,k,m}} \right)
    = 0 
\end{align*}
in probability.
\end{proposition}

\begin{proof}
We apply Lemma \ref{Lem-preliminary-compare} for the matrix $A = Y \sqrt{\Lambda_m}$ and $B = \widetilde Y \sqrt{\Lambda_m}$, where $\sqrt{\Lambda_m} = \diag(\sqrt{\tau_1},\ldots,\sqrt{\tau_m})$. We have
\begin{align} \label{ineq-difference}
    L^4 \left( F^{M_{n,k,m}}, F^{\widetilde M_{n,k,m}} \right)
    \le \dfrac{2}{n^{2k}} \Tr \left( \left( Y - \widetilde Y \right) \Lambda_m \left( Y - \widetilde Y \right)^* \right) \Tr \left( Y \Lambda_m Y^* + \widetilde Y \Lambda_m \widetilde Y^* \right).
\end{align}
We handle the right hand side of \eqref{ineq-difference} term by term.

\bigskip

\noindent{\bf Step 1.} In this step, we bound the second term
\begin{align*}
    \dfrac{1}{n^k} \Tr \left( Y \Lambda_m Y^* + \widetilde Y \Lambda_m \widetilde Y^* \right)
\end{align*}
in \eqref{ineq-difference}.

By \eqref{eq-correlated matrix}, we have
\begin{align} \label{ineq-trace-bound-1}
    \Tr \left( Y \Lambda_m Y^* \right)
    = \sum_{\alpha=1}^m \tau_\alpha \dfrac{\Tr \left( Y_{\alpha} Y_{\alpha}^* \right)}{\|Y_\alpha\|^2}
    = \sum_{\alpha=1}^m \tau_\alpha \dfrac{\Tr \left( Y_{\alpha}^* Y_{\alpha} \right)}{\|Y_\alpha\|^2}
    = \sum_{\alpha=1}^m \tau_\alpha.
\end{align}
By \eqref{eq-covariance matrix}, we get
\begin{align*}
    \Tr \left( \widetilde Y \Lambda_m \widetilde Y^* \right)
    = \dfrac{1}{n^k} \sum_{\alpha=1}^m \tau_{\alpha} \Tr \left( Y_{\alpha} Y_{\alpha}^* \right)
    = \dfrac{1}{n^k} \sum_{\alpha=1}^m \tau_{\alpha} \|Y_{\alpha}\|^2,
\end{align*}
which together with Lemma \ref{Lem-moment-Y} implies
\begin{align*}
    \bE \left[ \Tr \left( \widetilde Y \Lambda_m \widetilde Y^* \right) \right]
    = \sum_{\alpha=1}^m \tau_\alpha,
\end{align*}
and
\begin{align*}
    &\bE \left[ \left| \Tr \left( \widetilde Y \Lambda_m \widetilde Y^* \right) \right|^2 \right]
    = \dfrac{1}{n^{2k}} \left( \sum_{\alpha=1}^m \tau_\alpha^2 \bE \left[ \|Y_\alpha\|^4 \right] + \sum_{\alpha\not=\beta,\alpha,\beta=1}^m \tau_\alpha\tau_\beta \bE \left[ \|Y_\alpha\|^2 \|Y_\beta\|^2 \right] \right) \\
    =& \sum_{\alpha=1}^m \tau_\alpha^2 \left( 1 + \dfrac{m_4-1}{n} \right)^k + \sum_{\alpha\not=\beta,\alpha,\beta=1}^m \tau_\alpha\tau_\beta
    = \left( \sum_{\alpha=1}^m \tau_\alpha \right)^2 + \sum_{\alpha=1}^m \tau_\alpha^2 \left( \left( 1 + \dfrac{m_4-1}{n} \right)^k - 1 \right).
\end{align*}
Hence, for any $x>0$, by Markov inequality and Lemma \ref{Lem-moment-Y}, we obtain
\begin{align} \label{ineq-trace-bound-2}
    &\bP \left( \left| \Tr \left( \widetilde Y \Lambda_m \widetilde Y^* \right) - \sum_{\alpha=1}^m \tau_\alpha \right| > x \right) \nonumber \\
    \le& \dfrac{1}{x^2} \bE \left[ \left| \Tr \left( \widetilde Y \Lambda_m \widetilde Y^* \right) - \sum_{\alpha=1}^m \tau_\alpha \right|^2 \right]
    = \dfrac{1}{x^2} \left( \bE \left[ \left| \Tr \left( \widetilde Y \Lambda_m \widetilde Y^* \right) \right|^2 \right] - \left( \sum_{\alpha=1}^m \tau_\alpha \right)^2 \right) \nonumber \\
    =& \dfrac{1}{x^2} \sum_{\alpha=1}^m \tau_\alpha^2 \left( \left( 1 + \dfrac{m_4-1}{n} \right)^k - 1 \right).
\end{align}
Now we choose $x = n^k$ in \eqref{ineq-trace-bound-2} and use \eqref{ineq-trace-bound-1} to obtain
\begin{align*}
    \bP \left( \left| \dfrac{1}{n^k} \Tr \left( Y \Lambda_m Y^* + \widetilde Y \Lambda_m \widetilde Y^* \right) - \dfrac{2}{n^k} \sum_{\alpha=1}^m \tau_\alpha \right| > 1 \right)
    \le \dfrac{1}{n^{2k}} \sum_{\alpha=1}^m \tau_\alpha^2 \left( \left( 1 + \dfrac{m_4-1}{n} \right)^k - 1 \right).
\end{align*}
By Taylor expansion, when $n \to \infty$, we have
\begin{align} \label{eq-appro}
    & \left( \left( 1 + \dfrac{m_4-1}{n} \right)^k - 1 \right)
    = \left( \exp \left( k \ln \left( 1 + \dfrac{m_4-1}{n} \right) \right) - 1 \right) \nonumber \\
    =& \left( \exp \left( (m_4-1) \dfrac{k}{n} + O(kn^{-2}) \right) - 1 \right)
    = \left( (m_4-1) \dfrac{k}{n} + O(k^2n^{-2}) \right),
\end{align}
which together with the convergence \eqref{eq-def-tau} implies
\begin{align*}
    \sum_{n=1}^\infty \bP \left( \left| \dfrac{1}{n^k} \Tr \left( Y \Lambda_m Y^* + \widetilde Y \Lambda_m \widetilde Y^* \right) - \dfrac{2}{n^k} \sum_{\alpha=1}^m \tau_\alpha \right| > 1 \right) < \infty.
\end{align*}
By Borel-Cantelli's lemma, with probability one, it holds eventually that
\begin{align} \label{ineq-trace-bound-3}
    \left| \dfrac{1}{n^k} \Tr \left( Y \Lambda_m Y^* + \widetilde Y \Lambda_m \widetilde Y^* \right) - \dfrac{2}{n^k} \sum_{\alpha=1}^m \tau_\alpha \right| \le 1.
\end{align}

\bigskip

\noindent{\bf Step 2.} In this step, we establish the convergence of the first term
\begin{align*}
    \dfrac{1}{n^k} \Tr \left( \left( Y - \widetilde Y \right) \Lambda_m \left( Y - \widetilde Y \right)^* \right)
\end{align*}
in \eqref{ineq-difference}.

We apply Lemma \ref{Lem-matrix} with $A = (Y_1 \ldots Y_m) \sqrt{\Lambda_m}$, we have
\begin{align} \label{ineq-trace-bound-4}
    \Tr \left( \left( Y - \widetilde Y \right) \Lambda_m \left( Y - \widetilde Y \right)^* \right)
    = \sum_{\alpha=1}^m \tau_\alpha \left( \dfrac{\|Y_\alpha\|^2}{n^k} - 1 \right) - 2 \sum_{\alpha=1}^m \tau_\alpha \left( \dfrac{\|Y_\alpha\|}{\sqrt n^k} - 1 \right).
\end{align}

On one hand, we have
\begin{align} \label{ineq-estimate-1}
    & \left| \sum_{\alpha=1}^m \tau_\alpha \left( \dfrac{\|Y_\alpha\|}{\sqrt n^k} - 1 \right) \right|
    \le \sum_{\alpha=1}^m \tau_\alpha \left| \dfrac{\|Y_\alpha\|}{\sqrt n^k} - 1 \right|
    \le \sum_{\alpha=1}^m \tau_\alpha \left| \dfrac{\|Y_\alpha\|^2}{n^k} - 1 \right|.
\end{align}
On the other hand, we can see that
\begin{align} \label{ineq-estimate-2}
    \left| \sum_{\alpha=1}^m \tau_\alpha \left( \dfrac{\|Y_\alpha\|^2}{n^k} - 1 \right) \right|
    \le \sum_{\alpha=1}^m \tau_\alpha \left| \dfrac{\|Y_\alpha\|^2}{n^k} - 1 \right|.
\end{align}
Substituting \eqref{ineq-estimate-1} and \eqref{ineq-estimate-2} to \eqref{ineq-trace-bound-4}, we get
\begin{align} \label{ineq-trace-bound-5}
    \left| \Tr \left( \left( Y - \widetilde Y \right) \Lambda_m \left( Y - \widetilde Y \right)^* \right) \right|
    \le 3 \sum_{\alpha=1}^m \tau_\alpha \left| \frac{\|Y_\alpha\|^2}{n^k} - 1 \right|.
\end{align}

By Cauchy-Schwarz inequality, Lemma \ref{Lem-moment-Y} and independence, we have
\begin{align*}
    & \bE \left[ \left( \sum_{\alpha=1}^m \tau_\alpha \left| \dfrac{\|Y_\alpha\|^2}{n^k} - 1 \right| \right)^2 \right] \\
    =& \bE \left[ \sum_{\alpha=1}^m \tau_\alpha^2 \left| \dfrac{\|Y_\alpha\|^2}{n^k} - 1 \right|^2 + \sum_{\alpha\not=\beta,\alpha,\beta=1}^m \tau_\alpha\tau_\beta \left| \dfrac{\|Y_\alpha\|^2}{n^k} - 1 \right| \left| \dfrac{\|Y_\beta\|^2}{n^k} - 1 \right| \right] \\
    \le& \sum_{\alpha=1}^m \tau_\alpha^2 \bE \left[ \left| \dfrac{\|Y_\alpha\|^2}{n^k} - 1 \right|^2 \right] + \sum_{\alpha\not=\beta,\alpha,\beta=1}^m \tau_\alpha\tau_\beta \left( \bE \left[ \left| \dfrac{\|Y_\alpha\|^2}{n^k} - 1 \right|^2 \right] \right)^{1/2} \left( \bE \left[ \left| \dfrac{\|Y_\beta\|^2}{n^k} - 1 \right|^2 \right] \right)^{1/2} \\
    =& \left( \sum_{\alpha=1}^m \tau_\alpha \right)^2 \bE \left[ \left| \dfrac{\|Y_\alpha\|^2}{n^k} - 1 \right|^2 \right]
    = \left( \sum_{\alpha=1}^m \tau_\alpha \right)^2 \left( \left( 1 + \dfrac{m_4-1}{n} \right)^k - 1 \right).
\end{align*}
Together with \eqref{ineq-trace-bound-5} and Markov inequality, we obtain that for any $x>0$,
\begin{align*}
    & \bP \left( \left| \Tr \left( \left( Y - \widetilde Y \right) \Lambda_m \left( Y - \widetilde Y \right)^* \right) \right| > x \right)
    \le \bP \left( \sum_{\alpha=1}^m \tau_\alpha \left| \dfrac{\|Y_\alpha\|^2}{n^k} - 1 \right| > \dfrac{x}{3} \right) \\
    \le& \dfrac{9}{x^2} \bE \left[ \left( \sum_{\alpha=1}^m \tau_\alpha \left| \dfrac{\|Y_\alpha\|^2}{n^k} - 1 \right| \right)^2 \right]
    \le \dfrac{9}{x^2} \left( \sum_{\alpha=1}^m \tau_\alpha \right)^2 \left( \left( 1 + \dfrac{m_4-1}{n} \right)^k - 1 \right).
\end{align*}
We choose $x = n^k y$ and use \eqref{eq-appro} to get
\begin{align} \label{ineq-trace-bound-5'}
    \bP \left( \dfrac{1}{n^k} \left| \Tr \left( \left( Y - \widetilde Y \right) \Lambda_m \left( Y - \widetilde Y \right)^* \right) \right| > y \right)
    \le \left( \dfrac{1}{n^k} \sum_{\alpha=1}^m \tau_\alpha \right)^2 \dfrac{9}{y^2} \left( (m_4-1) \dfrac{k}{n} + O(k^2n^{-2}) \right).
\end{align}
In view of \eqref{eq-lim-k}, there exists $y = y(n)$ satisfying
\begin{align*}
    \lim_{n \to \infty} \dfrac{1}{y^2} \dfrac{k}{n} = 0
    \quad \mathrm{and} \quad
    \lim_{n \to \infty} y = 0.
\end{align*}
With this $y$, \eqref{ineq-trace-bound-5'} implies that
\begin{align} \label{ineq-trace-bound-6}
    \lim_{n \to \infty} \dfrac{1}{n^k} \left| \Tr \left( \left( Y - \widetilde Y \right) \Lambda_m \left( Y - \widetilde Y \right)^* \right) \right| = 0
\end{align}
in probability, noting the convergence \eqref{eq-def-tau}.

\bigskip

The proof is concluded by substituting \eqref{ineq-trace-bound-3} and \eqref{ineq-trace-bound-6} to \eqref{ineq-difference}.
\end{proof}

Now we are ready to prove Theorem \ref{Thm}.

\begin{proof}(of Theorem \ref{Thm}.)
By triangle inequality, Lemma \ref{Lem-LSD-M} and Proposition \ref{Prop-compare}, we obtain that
\begin{align*}
    \lim_{n \to \infty} L \left( F^{M_{n,k,m}}, F_c \right)
    \le \lim_{n \to \infty} L \left( F^{\widetilde M_{n,k,m}}, F^{M_{n,k,m}} \right) + \lim_{n \to \infty} L \left( F^{\widetilde M_{n,k,m}}, F_c \right)
    = 0 
\end{align*}
in probability.
\end{proof}

Next, we turn to the proof of Corollary \ref{Coro}.

\begin{proof}(of Corollary \ref{Coro}.)
Let $\xi_1$ to be a standard Gaussian random variable.
By \eqref{eq-||Y||2}, we have
\begin{align*}
    \dfrac{Y_\alpha}{\|Y_\alpha\|}
    = \dfrac{\y_{\alpha}^{(1)} \otimes \cdots \otimes \y_{\alpha}^{(k)}}{\prod_{l=1}^k \|\y_\alpha^{(l)}\|}
    = \dfrac{\y_{\alpha}^{(1)}}{\|\y_\alpha^{(1)}\|} \otimes \cdots \otimes \dfrac{\y_{\alpha}^{(k)}}{\|\y_\alpha^{(k)}\|}.
\end{align*}
As the random vector $\y_{\alpha}^{(1)}/\|\y_{\alpha}^{(1)}\|$ is distributed on the unit sphere $S^{2n-1}$ of $\bC^n$, and is unitary invariant due to the Gaussian setting of $\xi_1$, one can easily deduce that $\y_{\alpha}^{(1)}/\|\y_{\alpha}^{(1)}\|$ is distributed on $S^{2n-1}$ uniformly.
Hence, we obtain that
\begin{align*}
    M_{n,k,m} = M_{n,k,m}'.
\end{align*}
in distribution.

The Gaussian assumption on $\xi_1$ implies that it has finite moment of any order. Hence, by Theorem \ref{Thm}, we obtain the convergence
\begin{align*}
    \lim_{n \to \infty} F^{M_{n,k,m}'} = \lim_{n \to \infty} F^{M_{n,k,m}} = F_c
\end{align*}
in probability.

\end{proof}

\section*{Acknowledgments} 

WY was supported by Guangdong Basic and Applied Basic Research Foundation (No. 2026A1515030040), and National Natural Science Foundation of China (No. 12501183), and a Grant of the Department of Science and Technology of Guangdong Province (No. 2024QN11X161).

\bibliographystyle{plain}
\bibliography{tensor}

\end{document}